\documentclass{amsart}
\usepackage{geometry} 
\usepackage{amsthm}
\usepackage{amsmath}
\usepackage{amssymb}
\usepackage[usenames,dvipsnames]{xcolor}
\usepackage{gnuplot-lua-tikz}
\usepackage{tikz}
\pgfdeclarelayer{background layer}
\pgfsetlayers{background layer,main}
\tikzset{
   highlight/.style={line width=.5cm,color=black!22,cap=round,join=round,opacity=0.3}
} 
\usepackage[ruled,vlined,linesnumbered,norelsize]{algorithm2e}
\usepackage{hyperref}

\newcommand{\ol}[1]{\overline{#1}}

\newtheorem{theorem}{Theorem}[section]
\newtheorem{prop}[theorem]{Proposition}

\newtheorem{definition}[theorem]{Definition}
\newtheorem{example}{Example}
\newtheorem{fact}{Fact}

\DeclareMathOperator{\sw}{sw}
\DeclareMathOperator{\Geo}{Geo}
\DeclareMathOperator{\diam}{diam}
\newcommand{\N}{\mathbb{N}} 
\newcommand{\Z}{\mathbb{Z}}
\newcommand{\cS}{{\mathcal S}}

\newcommand{\BioGAP}{\textsc{BioGAP}}
\newcommand{\GAP}{\textsc{GAP}}

\usepackage{url}

\title{Bacterial phylogeny in the Cayley graph}
\author{Chad Clark, Attila Egri-Nagy, Andrew R.~Francis, Volker Gebhardt}
\address{Centre for Research in Mathematics, School of Computing, Engineering and Mathematics, Western Sydney University, Locked Bag 1797, Penrith, NSW 2751, Australia}

\email{\texttt{17485151@student.uws.edu.au}, \texttt{\{A.Egri-Nagy,A.Francis,V.Gebhardt\}\-@westernsydney.edu.au}}

\begin{document}
\maketitle
\begin{abstract}
Many models of genome rearrangement involve operations (e.g. inversions and translocations) that are self-inverse, and hence generate a group acting on the space of genomes.  This gives a correspondence between genome arrangements and the elements of a group, and consequently, between evolutionary paths and walks on the Cayley graph.  Many common methods for phylogeny reconstruction rely on calculating the minimal distance between two genomes; this omits much of the other information available from the Cayley graph.  In this paper we begin an exploration of some of this additional information, in particular describing the phylogeny as a Steiner tree within the Cayley graph, and exploring the ``interval'' between two genomes.  While motivated by problems in systematic biology, many of these ideas are of independent group-theoretic interest.
\end{abstract}

\section{Introduction}

Large-scale genome rearrangements are widely used for phylogenetic inference, especially in bacteria, because they occur frequently enough to use as a measure of evolutionary distance, but not so frequently that all signal is lost.  Each genome is then thought of as an arrangement of regions of DNA that are present in all the genomes under study.  Such regions need not be functional DNA (genes), but are identified through genome alignment software such as MAUVE~\cite{darling2004mauve}.  

The most commonly-studied rearrangement process in bacteria is \emph{inversion} (others include translocation, deletion, insertion, etc).  Inversion involves excision of a segment of the circular chromosome, followed by reinsertion in the same place but with orientation reversed.  This can thus be thought of as a permutation of the regions on the genome, or a signed permutation if one accounts for orientation.  

Because inversion has order 2, in particular is invertible, the set of allowable inversions generates a group that acts on the space of bacterial genomes.  Choices of allowable inversions give rise to numerous group-theoretic models (see~\cite{egri2014group}).  Given a choice of model and associated set of inversions, a sequence of inversions corresponds to a sequence of generators, and describes a walk on the Cayley graph of the group.  Once one chooses a reference genome, this genome corresponds to the identity element of the group, and every other genome corresponds to a unique group element given by the word in the generators on a path to the genome through the Cayley graph.  

In other words, by using the inversions as generators of a group, and by fixing a reference genome, we have a one-to-one correspondence between genomes and group elements.  Furthermore, we have a correspondence between paths on the Cayley graph and sequences of evolutionary events.  This means that \emph{the Cayley graph provides a map of genome space}.  

For phylogeny, because an actual series of evolutionary events is a path on the Cayley graph, the true phylogeny is a tree on the Cayley graph whose leaves are the extant genomes (assuming there is no homoplasy: the paths do not visit the same genome twice).  A parsimonious explanation of the evolutionary history is then a \emph{Steiner tree} connecting the genomes of interest in the Cayley graph.  

So, an alternative formulation of the phylogeny reconstruction problem is the problem of finding a Steiner tree in a Cayley graph.  This can be a very hard problem, as the genomes we consider typically have 60--80 regions, and so the group that we are interested in can have over $10^{100}$ elements. 

In this paper we begin to address this challenge in two directions.  First we study the ``interval'' between two group elements, which is a ranked poset describing the group elements on geodesics between the elements.  Second, we approach the Steiner tree problem by beginning with three group elements; this problem is also called the median problem.  We define the ``interior'' of three points, and show that it contains all median points (or Steiner points, or Fermat points).

Our definiton of an interval between two group elements (genomes) in Section~\ref{s:intervals} extends the traditional focus in genome rearrangements on finding the minimal distance between two genomes.  While currently it is possible to establish minimal distance very quickly for several models (see \cite{sankoff1992edit,hannenhalli1995transforming,bader2001linear} for the uniform distribution on inversions model, and \cite{egri2014group} for the 2-inversion model), this approach throws away a lot of information carried in the Cayley graph.  Traditional methods of studying such problems in genome rearrangements come from combinatorics~\cite{fertin2009combinatorics,gascuel2005mathematics} and from group theory~\cite{egri2014group,francis2014algebraic}.  Our main result on intervals (Theorem~\ref{t:main.intervals}) gives a condition under which two intervals are isomorphic, and comes after some exploration of various properties of intervals.

The median problem --- finding a point (genome rearrangement) that minimises the sum of the distances to a set of three other points --- has a long history in genomics and has been shown to be NP-hard for the most commonly studied model~\cite{caprara2003reversal} (though many heuristic solutions exist, e.g.~\cite{PevznerBourqe2002}).  In Section~\ref{s:median} we show (Prop.~\ref{p:interior}) that the median in a Cayley graph is in the interior of a triangle formed by the three group elements of interest.  The definition of the interior is a key part of the result, and this relies on the notion of interval, introduced in Section~\ref{s:intervals}.

We finish the paper (Section~\ref{s:algorithms}) with some algorithms for constructing the interval and some computational discussion of the problems. We use the \GAP~computer algebra system \cite{GAP4} and our \BioGAP~package \cite{biogap} to explore the combinatorics of these objects.

\section{Groups, generating sets, geodesic paths}

Let $G$ be a group with generators $S=\{s_1,\ldots,s_n\}$, so that every element of $G$ can be written as a finite product of generators and their inverses, denoted by $\langle S\rangle=G$. Sometimes we restrict these products to elements of $S$, not using their inverses, and we say that the (same) group is \emph{generated as a semigroup}. 
Let $S^*$ denote the free monoid generated by $S$, which is the set of all \emph{words}, i.e.\ finite sequences of the elements of $S$.
The empty word is denoted by $\varepsilon$.
The group element realized by the word $w$ is denoted by $\ol{w}$, thus $w\in S^*$ and $\ol{w}\in G$.

The \emph{geodesic distance} is defined by $d_S(g_1,g_2)=|u|$,
where $u$ is a minimal length word in $S^*$, called a \emph{geodesic} word, with the property that $g_1\ol{u}=g_2$ as group elements. This is also denoted by $g_1\overset{u}{\longrightarrow}g_2$.
\begin{fact}\label{f:distance.invariant}
Distance is invariant under left multiplication. That is, for $g,g_1,g_2\in G$, $d(g_1,g_2)=d(gg_1,gg_2)$.
\end{fact}
\begin{proof}
If $u$ is a geodesic word such that $g_1\ol{u}=g_2$, then $gg_1\ol{u}=gg_2$.  This provides a one-to-one correspondence between paths $g_1\overset{u}{\longrightarrow}g_2$ and $gg_1\overset{u}{\longrightarrow}gg_2$.
\end{proof}

Let $\Geo_S(g_1,g_2)$ be the set of all geodesic words from $g_1$ to $g_2$, so that 
\[\Geo_S(g_1,g_2)=\{w\in S^*\mid g_1\ol{w}=g_2\text{ and }d_S(g_1,g_2)=|w|\}.\]
When no confusion arises  we omit the reference to the generating set and simply use $d(g_1,g_2)$ and $\Geo(g_1,g_2)$.
The length of a group element $g$ is defined by its distance from the identity: $\ell(g)=d(1_G,g)$ and we also write $\Geo(g)$ instead of $\Geo(1_G,g)$, where $1_G$ is the identity of the group.
In particular $\ell(1_G)=0$ and $\Geo(1_G)=\{\varepsilon\}$.
The \emph{diameter} of $G$ is defined by $\diam(G)=\max_{g\in G}{\ell(g)}$.

The \emph{Cayley graph} $\Gamma(G,S)$ of $G$ with respect to the generating set $S$ is the directed graph with group elements as nodes and the labelled edges encoding the action of $G$ on itself by its generators.
Thus, $g\overset{s}{\longrightarrow}gs$ is an edge, for $s\in S$. 

For a finite set $X$, let $\cS_X$ be the group of all permutations of $X$, the \emph{symmetric group on $X$}.

A \emph{partial order} is a reflexive, antisymmetric and transitive relation on a set, often denoted by the symbol $\leq$.
We write $x<y$ if $x\leq y$ and $x\neq y$.
The \emph{covering relation} is defined by $x\prec y$ if $x<y$ and there is no $z$ such that $x<z<y$.
An ordered set $X$ is an \emph{antichain} if for $x,y\in X$, $x\leq y$ implies $x=y$.
A map $\varphi$ is an \emph{order-isomorphism} if $x\leq y \Longleftrightarrow \varphi(x)\leq \varphi(y)$.
For more details on order theory see \cite{Lattices2ndEd,caspard2012finite}.

\section{Intervals}\label{s:intervals}

For a group $G$ we define a partial order of the group elements based on how they are ordered along geodesics with respect to a generating set $S$.
\begin{definition}[Prefix order] For group elements $g_1,g_2\in G=\langle S \rangle$ we write
$g_1\leq_S g_2$ if $\exists w=uv\in S^*$  such that  $\ol{w}=g_2, \ol{u}=g_1, w\in \Geo(g_2)$.
That is, $g_1\le_S g_2$ if there is a geodesic from the identity to $g_2$ such that $g_1$ is on it.
\end{definition}
\noindent We just write $\leq$ instead of $\leq_S$, if no confusion arises.

For Coxeter systems this partial order is called the \emph{weak order} \cite{bjorner2005combinatorics}, and it is a ranked (graded) lattice.
It is natural to ask that what are the properties of the weak order in general?

\begin{definition}[Ranking]
An ordered set $(X,\leq)$ is \emph{ranked} if there is a rank function $r:X\rightarrow (\N,\leq)$ preserving the covering relation:
$$x\prec y\Rightarrow r(y)=r(x)+1.$$
\end{definition}

\begin{prop}
Any group $G$ ordered by the prefix order $(G,\leq)$ is a ranked poset with  $\ell$ as the rank function.
\end{prop}
\begin{proof}
$g_1\prec g_2$ means that $g_1$ is on a geodesic going to $g_2$ and there is no other element in between, therefore $\ell(g_2)=\ell(g_1)+1$.
\end{proof}

\begin{definition} The \emph{rank-set} $R_i$ is the set of elements of rank $i$:  $R_i=\{g\in G\mid\ell(g)=i\}$, $i\in\N$.
\end{definition}

The rank-sets are  easy to visualize for the simple and familiar example of $\Z^2$ with the standard generator set $\{x=(1,0),y=(0,1)\}$, as in Figure~\ref{fig:ZxZlayers}.
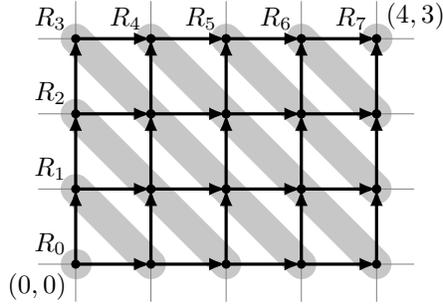
\begin{figure}
\begin{center}
\begin{tikzpicture}
\foreach \p in {0,1,2,3}{ 
  \draw[line width=0.4cm,color=black!22,cap=round,join=round] (\p,0)--(0,\p);
  \draw[above left]  (0,\p) node {$R_\p$};
  \pgfmathparse{\p+1} \let\q\pgfmathresult
  \draw[line width=0.4cm,color=black!22,cap=round,join=round] (\q,3)--(4,\p);
  \pgfmathparse{\p+4} \let\lab\pgfmathresult  
  \draw[above left]  (\q,3) node {$R_{\pgfmathprintnumber{\lab}}$};
}

  \draw[help lines] (-0.5,-0.5) grid (4.5,3.5);
  \fill (0,0) circle (0.06cm);
  \draw[below left] (0,0) node(n00){$(0,0)$};
  \draw[above right] (4,3) node(n43){$(4,3)$};

\foreach \x in {0,1,2,3,4}{
  \foreach \y in {0,1,2} {
    \pgfmathparse{\y+1} \let\height\pgfmathresult
    \draw[very thick,->,>=latex] (\x,\y)--(\x,\height); 
    \fill (\x,\height) circle (0.06cm);
  }
}
\foreach \y in {0,1,2,3}{
  \foreach \x in {0,1,2,3} {
    \pgfmathparse{\x+1} \let\w\pgfmathresult
    \draw[very thick,->,>=latex] (\x,\y)--(\w,\y); 
    \fill (\w,\y) circle (0.06cm);
  }
}

\end{tikzpicture}
\end{center}
\caption{The rank-sets of the interval $\big[(0,0),(4,3)\big]$ in $\Z^2$.}
\label{fig:ZxZlayers}
\end{figure}

\begin{example}[Free Groups]
The Cayley graphs of free groups are trees, and so there is only one geodesic to each element.
The rank-sets are formed simply by words of the same length.
\end{example}

\begin{example}
Another way to have unique geodesics for every group element in a group $G$ is if the group itself is the generating set.
In the prefix order $(G,\leq_G)$ for each $g\in G$, $g\neq 1_G$ we have one geodesic, a single edge.  That is, all non-identity elements have length 1.
\end{example}

\begin{definition}[Interval] For a group $G$ with a prefix order determined by a fixed generating set, the closed \emph{intervals} are defined by
$$I_h:=[1_G,h]=\left(\{g\in G \mid 1_G\leq g\leq h \},\leq \right).$$
\end{definition}
Due to the symmetries of the Cayley graph the identity is only a convenient choice, we can consider intervals between arbitrary elements (see Fact~\ref{f:distance.invariant}).
\begin{fact}
$[g_1,g_2]$ is order-isomorphic to $[1_G,g_1^{-1}g_2]$.
\end{fact}

\begin{example}[Cyclic Groups]
$C_n=\langle x\rangle=\{x\mid x^n=1\}$ generated as a semigroup has only one geodesic to each element, thus its intervals are also linear orders. 
The situation changes if the cyclic group is generated as a group: $\diam_{\{x\}}(C_n)=n-1$ while $\diam_{\{x,x^{-1}\}}(C_n)=\lfloor\frac{n}{2}\rfloor$.
If $n$ is even, then $\big|\Geo\big(x^{\frac{n}{2}}\big)\big|=2$ and $\big[1,x^{\frac{n}{2}}\big]=C_n$. This shows that it is possible to have the entire group in an interval.
\end{example}

There are some immediate negative results about intervals in general.
To what extent does the interval depend on the distance between the elements? Is there a relation between $|\Geo(g)|$ and $\big|[1,g]\big|$?
In general, the length of the element does not determine the size of the interval uniquely. Again, $\Z^2$ demonstrates this easily (Fig.\ \ref{fig:ZxZlengthvswidth}).

\begin{figure}
\begin{center}
\begin{tikzpicture}
  \draw[line width=0.3cm,color=black!22,cap=round,join=round] (0,0)--(4,0); 
  \draw[help lines] (-0.5,-0.5) grid (4.5,1.5);
  \fill (0,0) circle (0.06cm);
  \draw[below left] (0,0) node(n00){$(0,0)$};
  \draw[below right] (4,0) node(n04){$(4,0)$};
  \foreach \x in {0,1,2,3} {
    \pgfmathparse{\x+1} \let\height\pgfmathresult
    \draw[very thick,->,>=latex] (\x,0)--(\height,0); 
    \fill (\height,0) circle (0.06cm);
  }
    \node  (x) at (0.5,-0.3) {$x$};

\end{tikzpicture}
\hskip1cm
\begin{tikzpicture}
  \fill[line width=0.3cm,color=black!22,cap=round,rounded corners] (-0.15,-0.15) rectangle (2.15,2.15); 
  \draw[help lines] (-0.5,-0.5) grid (2.5,2.5);
  \fill (0,0) circle (0.06cm);
  \draw[below left] (0,0) node(n00){$(0,0)$};
  \draw[above right] (2,2) node(n22){$(2,2)$};

\foreach \x in {0,1,2}{
  \foreach \y in {0,1} {
    \pgfmathparse{\y+1} \let\height\pgfmathresult
    \draw[very thick,->,>=latex] (\x,\y)--(\x,\height); 
    \draw[very thick,->,>=latex] (\y,\x)--(\height,\x); 
    \fill (\x,\height) circle (0.06cm);
    \fill (\height,\x) circle (0.06cm);
  }
}
  \node  (y) at (-0.3,0.5) {$y$};
      \node  (x) at (0.5,-0.3) {$x$};
\end{tikzpicture}
\end{center}
\caption{Two elements of $\Z^2$ with the same length (4), but intervals of different sizes: $\big|[(0,0),(4,0)]\big|=5$, while $\big|[(0,0),(2,2)]\big|=9$, and number of geodesics are 1 and 6 respectively.}
\label{fig:ZxZlengthvswidth}
\end{figure}
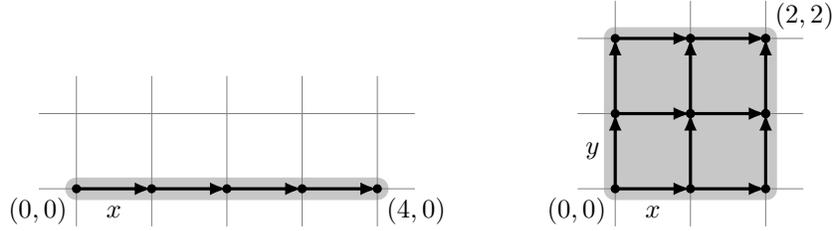

A partial order satisfies the Sperner property if no antichain is bigger than the size of the maximal rank-set~\cite{caspard2012finite}.
Intervals do not satisfy this property, and in general the number of paths can be bigger than the size of a maximal antichain, see Fig.\ \ref{fig:NonSperner}.
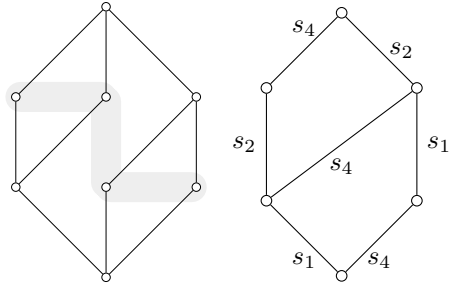
\begin{figure}
\begin{center}
\scalebox{0.8}{\begin{tikzpicture}
  \tikzstyle{every node}=[draw,circle,fill=white,minimum size=4pt,inner sep=0pt]
\draw (2,4.5) node (31) {};
\draw (.5,3) node (21) {};
\draw (2,3) node  (22) {};
\draw (3.5,3) node  (23) {};
\draw (.5,1.5) node (11) {};
\draw (2,1.5) node  (12) {};
\draw (3.5,1.5) node  (13) {};
\draw (2,0) node (01)  {};

\draw (01) -- (11);
\draw (01) -- (12);
\draw (01) -- (13);
\draw (11) -- (21);
\draw (11) -- (22);
\draw (12) -- (23);
\draw (13) -- (23);
\draw (31) -- (21);
\draw (31) -- (22);
\draw (31) -- (23);
\begin{pgfonlayer}{background layer}
\draw[highlight] (21)--(22)--(12)--(13);
\end{pgfonlayer}

\end{tikzpicture}}
\begin{tikzpicture}
  \tikzstyle{every node}=[draw,circle,fill=white,minimum size=4pt,inner sep=0pt]
  \tikzstyle{label}=[draw=none,fill=none,inner sep=2pt]

\draw (1,3.5) node (31) {};
\draw (0,2.5) node (21) {};
\draw (2,2.5) node (22)  {};
\draw (0,1)node (11) {};
\draw (2,1) node (12) {};
\draw (1,0) node (01)  {};

\draw (01) -- node [below,label] {$s_1$} (11);
\draw (01) -- node [below,label] {$s_4$}  (12);

\draw (11) -- node [left,label] {$s_2$} (21);
\draw (11) --  node [below,label] {$s_4$} (22);

\draw (12) --  node [right,label] {$s_1$} (22);

\draw (31) -- node [above,label] {$s_4$}(21);
\draw (31) -- node [right,label] {$s_2$} (22);

\end{tikzpicture}
\end{center}
\caption{\textbf{Left:} This example shows that intervals do not in general satisfy the Sperner property.  Here $\max_i|R_i|=3$, but the highlighted elements form an antichain of size 4. The interval encodes commuting transpositions in $\cS_4$: $s_1s_3s_2=s_3s_1s_2=s_4s_3s_1=s_4s_1s_3=(1,3,4,2)$, where we use a circular generating set, i.e.~$s_4=(4,1)$. \textbf{Right:} The interval $[(),(1,3,2)(4,5)]$ in $\cS_5$ generated by $\{s_1,\ldots,s_4\}$, where $s_i$ is the transposition $(i,i+1)$. The number of paths is 3 while the maximum size of an antichain is only 2.}
\label{fig:NonSperner}
\end{figure}

Lattices are partial orders where every pair of elements have a unique least and a greatest upper bound. In general  intervals in Cayley graphs are not lattices.
Any group with two generators satisfying $a^2=b^2$, $ab=ba$, both being non-identity and with no generator equivalent to these words will contain non-lattice intervals (see Fig.\ \ref{fig:nonlattice}). 
\begin{figure}
\begin{center}
\begin{tikzpicture}
  \tikzstyle{every node}=[draw,circle,fill=white,minimum size=4pt,inner sep=0pt]

\draw (1,3.5) node (31) {};
\draw (0,2.5) node (21) {};
\draw (2,2.5) node (22)  {};
\draw (0,1)node[fill=black] (11) {};
\draw (2,1) node[fill=black] (12) {};
\draw (1,0) node (01)  {};

\draw (01) -- (11);
\draw (01) -- (12);

\draw (11) -- (21);
\draw (11) -- (22);

\draw (12) -- (21);
\draw (12) -- (22);

\draw (31) -- (21);
\draw (31) -- (22);

\end{tikzpicture}
\hskip.5cm
\begin{tikzpicture}
  \tikzstyle{every node}=[draw,circle,fill=white,minimum size=4pt,inner sep=0pt]
  \tikzstyle{label}=[draw=none,fill=none,inner sep=2pt]

\draw (1,3.5) node (31) {};
\draw (0,2.5) node (21) {};
\draw (2,2.5) node (22)  {};
\draw (0,1)node (11) {};
\draw (2,1) node (12) {};
\draw (1,0) node (01)  {};

\draw (01) -- node [below,label] {$a$} (11) ;
\draw (01) -- node [below,label] {$b$} (12);

\draw (11) -- node [left,label] {$a$} (21);
\draw (11) -- node [above, very near start,label] {$b$} (22);

\draw (12) -- node [above, very near start,label] {$b$} (21);
\draw (12) -- node [right,label] {$a$} (22);

\draw (31) -- (21);
\draw (31) -- (22);

\end{tikzpicture}
\end{center}
\caption{\textbf{Left:} Minimal example of a non-lattice graded poset. The pair of filled elements have have two least upper bounds. \textbf{Right:} Labelling with two generators for possible non-lattice intervals. For instance $[(),(1,2)]$ with generators $(3,4)$ and $(1,2)(3,4)$.} 
\label{fig:nonlattice}
\end{figure}
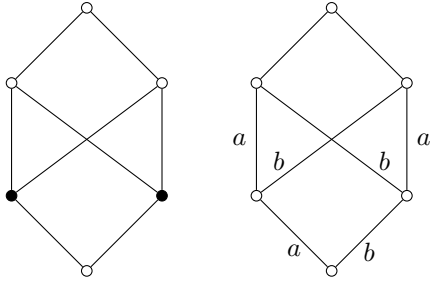

\section{Classifying group elements}\label{s:classifying}

The geodesic distance from a chosen reference point gives a
classification of group elements, the spheres around the reference.
Referring to properties of  intervals, we can have refined or
different classifications:

\begin{enumerate}
\item Same length: $g_1\sim_\ell g_2\Leftrightarrow \ell(g_1)=\ell(g_2)$,
\item Same number of paths: $g_1\sim_\# g_2\Leftrightarrow
  |\Geo(g_1)|=|\Geo(g_2)|$,
\item Same number of elements in intervals: $g_1\sim_{\mid\mid} g_2\Leftrightarrow |[1,g_1]|=|[1,g_2]|$,
\item Same interval: $g_1\sim_I g_2\Leftrightarrow [1,g_1]\cong[1,g_2]$, where $\cong$ denotes order isomorphism of partially ordered sets.
\end{enumerate}

\begin{figure}
\begin{center}
\begin{tikzpicture}[gnuplot]
\path (0.000,0.000) rectangle (14.000,7.000);
\gpcolor{color=gp lt color border}
\gpsetlinetype{gp lt border}
\gpsetlinewidth{1.00}
\draw[gp path] (1.688,0.985)--(1.868,0.985);
\draw[gp path] (13.447,0.985)--(13.267,0.985);
\node[gp node right] at (1.504,0.985) { 0};
\draw[gp path] (1.688,1.612)--(1.868,1.612);
\draw[gp path] (13.447,1.612)--(13.267,1.612);
\node[gp node right] at (1.504,1.612) { 1000};
\draw[gp path] (1.688,2.240)--(1.868,2.240);
\draw[gp path] (13.447,2.240)--(13.267,2.240);
\node[gp node right] at (1.504,2.240) { 2000};
\draw[gp path] (1.688,2.867)--(1.868,2.867);
\draw[gp path] (13.447,2.867)--(13.267,2.867);
\node[gp node right] at (1.504,2.867) { 3000};
\draw[gp path] (1.688,3.494)--(1.868,3.494);
\draw[gp path] (13.447,3.494)--(13.267,3.494);
\node[gp node right] at (1.504,3.494) { 4000};
\draw[gp path] (1.688,4.122)--(1.868,4.122);
\draw[gp path] (13.447,4.122)--(13.267,4.122);
\node[gp node right] at (1.504,4.122) { 5000};
\draw[gp path] (1.688,4.749)--(1.868,4.749);
\draw[gp path] (13.447,4.749)--(13.267,4.749);
\node[gp node right] at (1.504,4.749) { 6000};
\draw[gp path] (1.688,5.376)--(1.868,5.376);
\draw[gp path] (13.447,5.376)--(13.267,5.376);
\node[gp node right] at (1.504,5.376) { 7000};
\draw[gp path] (1.688,6.004)--(1.868,6.004);
\draw[gp path] (13.447,6.004)--(13.267,6.004);
\node[gp node right] at (1.504,6.004) { 8000};
\draw[gp path] (1.688,6.631)--(1.868,6.631);
\draw[gp path] (13.447,6.631)--(13.267,6.631);
\node[gp node right] at (1.504,6.631) { 9000};
\draw[gp path] (1.688,0.985)--(1.688,1.165);
\draw[gp path] (1.688,6.631)--(1.688,6.451);
\node[gp node center] at (1.688,0.677) { 0};
\draw[gp path] (3.158,0.985)--(3.158,1.165);
\draw[gp path] (3.158,6.631)--(3.158,6.451);
\node[gp node center] at (3.158,0.677) { 2};
\draw[gp path] (4.628,0.985)--(4.628,1.165);
\draw[gp path] (4.628,6.631)--(4.628,6.451);
\node[gp node center] at (4.628,0.677) { 4};
\draw[gp path] (6.098,0.985)--(6.098,1.165);
\draw[gp path] (6.098,6.631)--(6.098,6.451);
\node[gp node center] at (6.098,0.677) { 6};
\draw[gp path] (7.568,0.985)--(7.568,1.165);
\draw[gp path] (7.568,6.631)--(7.568,6.451);
\node[gp node center] at (7.568,0.677) { 8};
\draw[gp path] (9.037,0.985)--(9.037,1.165);
\draw[gp path] (9.037,6.631)--(9.037,6.451);
\node[gp node center] at (9.037,0.677) { 10};
\draw[gp path] (10.507,0.985)--(10.507,1.165);
\draw[gp path] (10.507,6.631)--(10.507,6.451);
\node[gp node center] at (10.507,0.677) { 12};
\draw[gp path] (11.977,0.985)--(11.977,1.165);
\draw[gp path] (11.977,6.631)--(11.977,6.451);
\node[gp node center] at (11.977,0.677) { 14};
\draw[gp path] (13.447,0.985)--(13.447,1.165);
\draw[gp path] (13.447,6.631)--(13.447,6.451);
\node[gp node center] at (13.447,0.677) { 16};
\draw[gp path] (1.688,6.631)--(1.688,0.985)--(13.447,0.985)--(13.447,6.631)--cycle;
\node[gp node center,rotate=-270] at (0.246,3.808) {frequency};
\node[gp node center] at (7.567,0.215) {geodesic length};
\node[gp node right] at (11.979,6.297) {S8};
\gpfill{rgb color={0.000,0.000,0.000},opacity=0.50} (12.163,6.220)--(13.079,6.220)--(13.079,6.374)--(12.163,6.374)--cycle;
\gpcolor{rgb color={0.000,0.000,0.000}}
\gpsetlinetype{gp lt plot 0}
\draw[gp path] (12.163,6.220)--(13.079,6.220)--(13.079,6.374)--(12.163,6.374)--cycle;
\gpfill{rgb color={0.000,0.000,0.000},opacity=0.50} (1.688,0.985)--(2.056,0.985)--(2.056,0.987)--(1.688,0.987)--cycle;
\draw[gp path] (1.688,0.985)--(1.688,0.986)--(2.055,0.986)--(2.055,0.985)--cycle;
\gpfill{rgb color={0.000,0.000,0.000},opacity=0.50} (2.055,0.985)--(2.791,0.985)--(2.791,0.991)--(2.055,0.991)--cycle;
\draw[gp path] (2.055,0.985)--(2.055,0.990)--(2.790,0.990)--(2.790,0.985)--cycle;
\gpfill{rgb color={0.000,0.000,0.000},opacity=0.50} (2.790,0.985)--(3.526,0.985)--(3.526,1.009)--(2.790,1.009)--cycle;
\draw[gp path] (2.790,0.985)--(2.790,1.008)--(3.525,1.008)--(3.525,0.985)--cycle;
\gpfill{rgb color={0.000,0.000,0.000},opacity=0.50} (3.525,0.985)--(4.261,0.985)--(4.261,1.061)--(3.525,1.061)--cycle;
\draw[gp path] (3.525,0.985)--(3.525,1.060)--(4.260,1.060)--(4.260,0.985)--cycle;
\gpfill{rgb color={0.000,0.000,0.000},opacity=0.50} (4.260,0.985)--(4.996,0.985)--(4.996,1.193)--(4.260,1.193)--cycle;
\draw[gp path] (4.260,0.985)--(4.260,1.192)--(4.995,1.192)--(4.995,0.985)--cycle;
\gpfill{rgb color={0.000,0.000,0.000},opacity=0.50} (4.995,0.985)--(5.731,0.985)--(5.731,1.483)--(4.995,1.483)--cycle;
\draw[gp path] (4.995,0.985)--(4.995,1.482)--(5.730,1.482)--(5.730,0.985)--cycle;
\gpfill{rgb color={0.000,0.000,0.000},opacity=0.50} (5.730,0.985)--(6.466,0.985)--(6.466,2.063)--(5.730,2.063)--cycle;
\draw[gp path] (5.730,0.985)--(5.730,2.062)--(6.465,2.062)--(6.465,0.985)--cycle;
\gpfill{rgb color={0.000,0.000,0.000},opacity=0.50} (6.465,0.985)--(7.201,0.985)--(7.201,3.042)--(6.465,3.042)--cycle;
\draw[gp path] (6.465,0.985)--(6.465,3.041)--(7.200,3.041)--(7.200,0.985)--cycle;
\gpfill{rgb color={0.000,0.000,0.000},opacity=0.50} (7.200,0.985)--(7.936,0.985)--(7.936,4.350)--(7.200,4.350)--cycle;
\draw[gp path] (7.200,0.985)--(7.200,4.349)--(7.935,4.349)--(7.935,0.985)--cycle;
\gpfill{rgb color={0.000,0.000,0.000},opacity=0.50} (7.935,0.985)--(8.671,0.985)--(8.671,5.588)--(7.935,5.588)--cycle;
\draw[gp path] (7.935,0.985)--(7.935,5.587)--(8.670,5.587)--(8.670,0.985)--cycle;
\gpfill{rgb color={0.000,0.000,0.000},opacity=0.50} (8.670,0.985)--(9.406,0.985)--(9.406,6.098)--(8.670,6.098)--cycle;
\draw[gp path] (8.670,0.985)--(8.670,6.097)--(9.405,6.097)--(9.405,0.985)--cycle;
\gpfill{rgb color={0.000,0.000,0.000},opacity=0.50} (9.405,0.985)--(10.141,0.985)--(10.141,5.360)--(9.405,5.360)--cycle;
\draw[gp path] (9.405,0.985)--(9.405,5.359)--(10.140,5.359)--(10.140,0.985)--cycle;
\gpfill{rgb color={0.000,0.000,0.000},opacity=0.50} (10.140,0.985)--(10.876,0.985)--(10.876,3.657)--(10.140,3.657)--cycle;
\draw[gp path] (10.140,0.985)--(10.140,3.656)--(10.875,3.656)--(10.875,0.985)--cycle;
\gpfill{rgb color={0.000,0.000,0.000},opacity=0.50} (10.875,0.985)--(11.611,0.985)--(11.611,2.005)--(10.875,2.005)--cycle;
\draw[gp path] (10.875,0.985)--(10.875,2.004)--(11.610,2.004)--(11.610,0.985)--cycle;
\gpfill{rgb color={0.000,0.000,0.000},opacity=0.50} (11.610,0.985)--(12.346,0.985)--(12.346,1.179)--(11.610,1.179)--cycle;
\draw[gp path] (11.610,0.985)--(11.610,1.178)--(12.345,1.178)--(12.345,0.985)--cycle;
\gpfill{rgb color={0.000,0.000,0.000},opacity=0.50} (12.345,0.985)--(13.081,0.985)--(13.081,1.005)--(12.345,1.005)--cycle;
\draw[gp path] (12.345,0.985)--(12.345,1.004)--(13.080,1.004)--(13.080,0.985)--cycle;
\gpfill{rgb color={0.000,0.000,0.000},opacity=0.50} (13.080,0.985)--(13.448,0.985)--(13.448,0.987)--(13.080,0.987)--cycle;
\draw[gp path] (13.080,0.985)--(13.080,0.986)--(13.447,0.986)--(13.447,0.985)--cycle;
\gpcolor{color=gp lt color border}
\gpsetlinetype{gp lt border}
\draw[gp path] (1.688,6.631)--(1.688,0.985)--(13.447,0.985)--(13.447,6.631)--cycle;
\gpdefrectangularnode{gp plot 1}{\pgfpoint{1.688cm}{0.985cm}}{\pgfpoint{13.447cm}{6.631cm}}
\end{tikzpicture}
\end{center}
\caption{Distribution of minimal lengths in $\cS_8$ generated by the
  circular generating set of transpositions.}
\label{fig:S8length}
\end{figure}
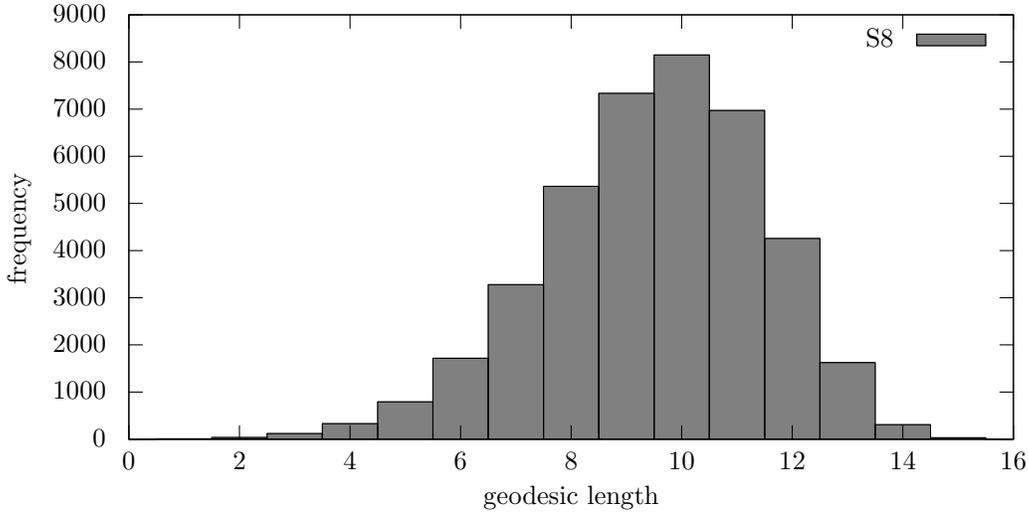
\begin{figure}
\begin{center}
\input{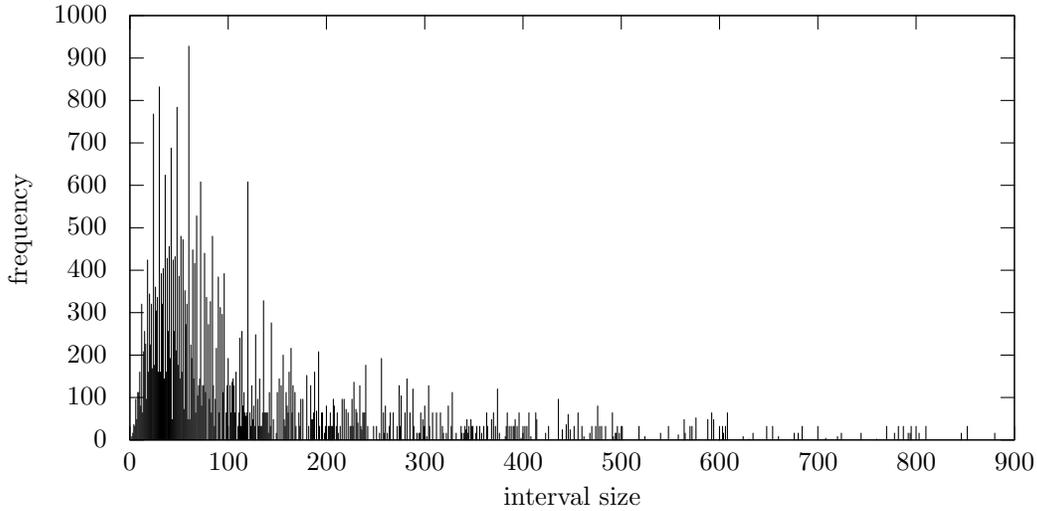}
\end{center}
\caption{Distribution of interval sizes in $\cS_8$ generated by the
  circular generating set of transpositions. There are 386 different interval lengths. Sporadic values are cut off, the maximum interval size is 4280.}
\label{fig:S8interval}
\end{figure}
The more properties of intervals we take into account the more refined the classifications are (see Fig.~\ref{fig:S8length} and \ref{fig:S8interval}). It is important to note, that the equivalence classes in different classifications are not in a single refinement hierarchy.
For instance, intervals of different length can have the same number of paths or same number of elements.

\subsection{Symmetries of generating sets and interval classification}
In practice, we would like to do the above classifications without fully calculating the potentially large intervals. This involves estimating or partially calculating some properties of an interval, just by looking at the permutation and the generating set. In some cases, we can easily decide equivalence.

An easy to calculate property of a permutation is its cycle-structure.
Unfortunately it is not true in general that conjugate elements have the same interval.
For instance, the symmetric group $\cS_5$ generated by the transpositions $\{s_1, \ldots, s_{4}\}$ has different lengths (thus different intervals) for $(1,3)$ and $(2,5)$, though they are conjugate elements in $\cS_5$. 
On the other hand, $(1,3)$ and $(3,5)$ have the same interval.
Closer inspection reveals that the latter two are conjugate under the flip symmetry of the `number line' segment, the sequence $1,2,3,4,5$, and thus the nontrivial symmetry the generator set has.
This suggests that if the generator set has some symmetries, then those symmetries will be retained by the intervals as well. In other words, if the generators are closed under conjugation, then elements conjugate by the same symmetries  have the same interval.

\begin{definition}[Normaliser]
The \emph{normaliser} of $S\subseteq\cS_X$ in $\cS_X$ is defined by $$N_{\cS_X}(S)=\{\sigma\in \cS_X\mid \sigma^{-1}S\sigma= S\}.$$
\end{definition}

\begin{example}
It is possible to have generating sets that are non-homogeneous in the sense that not all elements are conjugate, but the still share the same set of symmetries. Given 7 regions, and the generating set consisting of unsigned inversions of length 3,  $\{ (1,3), (2,4), (3,5), (4,6), (5,7), (1,6), (2,7)\}$,  and of length 4 $\{(1,4)(2,3), (2,5)(3,4), (3,6)(4,5), (4,7)(5,6), (1,5)(6,7), (1,7)(2,6), (1,2)(3,7)\}$ generate $\cS_7$.
This generating set has the dihedral group as its symmetries, but a 3-inversion is not conjugate to a 4-inversion.
\end{example}


Now we can give a sufficient condition for elements of a group to have order isomorphic intervals, as follows.

\begin{theorem}\label{t:main.intervals}
Let $G$ be a group acting on a set $X$ generated by $S$.
Write $N$ for the normaliser $N_{G}(S)$. 
For $g_1,g_2\in G$, we have
\[g_1\sim_N g_2\quad\implies\quad I_{g_1}\sim I_{g_2}
\]
where the $\sim_N$ on the left means ``conjugate under an element of $N$'' and $\sim$ on the right means ``order isomorphism''.
\end{theorem}

\begin{proof}
If $g_1\sim_N g_2$ then there is a $\pi\in N$ such that $\pi^{-1}g_1\pi=g_2$.  Take any edge in the interval $I_{g_1}$, say $w\overset{s}{\to} ws$ for $s\in S$.  Then
\[\pi^{-1}w\pi\overset{\pi^{-1}s\pi}{\longrightarrow}\pi^{-1}w\pi\cdot\pi^{-1}s\pi=\pi^{-1}ws\pi\]
as $\pi^{-1}s\pi$ is a generator in $S$ since $\pi\in N$.  That is, conjugation by elements of $N$ takes edges of $I_{g_1}$ to edges of $I_{g_2}$, defining an order embedding (since we can extend the above to show that $(u\le w)\implies (\pi^{-1}u\pi\le \pi^{-1}w\pi)$ for any $\pi\in N$.

To show this is an order isomorphism it remains to show it's a bijection on the underlying sets of $I_{g_1}$ and $I_{g_2}$.  This follows easily from the definition.  For one to one, if $\pi^{-1}w\pi=\pi^{-1}u\pi$ then $w=u$ since we are in a group.  For onto, the element $v\in I_{g_2}$ has pre-image $\pi v\pi^{-1}\in I_{g_1}$.  
\end{proof}

\begin{example}
A group generated by itself with the prefix order $(G,\leq_G)$ shows that isomorphic interval structure does not imply conjugacy.
\end{example}

\section{The Median Problem}\label{s:median}

In the Cayley graph $\Gamma(G,S)$, a solution to the median problem for $1_G, g_1, g_2 \in G$ corresponds to the vertex of degree $3$ in a Steiner tree connecting them. Note that since $G$ acts vertex-transitively on its Cayley graph by left multiplication, the median problem on any set of three group elements can be translated back to this problem in which one element is the identity. 
As a computationally difficult problem, it is useful to narrow the search for this vertex to a smaller subset of Cayley graph called the interior. To do this, we use the intervals between each pair of genomes. 

For group elements $1_G$, $g_1$ and $g_2$ begin by constructing the intervals $I_{g_1}$, $I_{g_2}$ and $[g_1, g_2]$.
Next, calculate the distances  $d(1_G,[g_1,g_2]) := \min\{ d(1_G,x) : x\in [g_1,g_2]\}$, $d\big(g_1,I_{g_2}\big)$ and $d\big(g_2,I_{g_1}\big)$. Denote these distance $\delta(1_G)$, $\delta(g_1)$ and $\delta(g_2)$ respectively. Finally, for $g \in G$ and $\delta\in\N$, denote by $\ol{B}_{\delta}(g)$ the set $\{h \in G \mid d(g,h) \leq \delta\}$, the \emph{closed ball} of radius $\delta$.

\begin{definition}[Interior] For group elements $1_G$, $g_1$ and $g_2$, their interior is the intersection 
$$\ol{B}_{\delta(1_G)}(1_G)\cap \ol{B}_{\delta(g_1)}(g_1)\cap \ol{B}_{\delta(g_2)}(g_2).$$

\end{definition}

Let $p(h_1,h_2)$ be a geodesic path between $h_1$ and $h_2$ in $G$. If $h_2 = 1_G$, this is written $p(h_1)$. In choosing the geodesic $p(g_2) \in \Geo(g_2)$ such that $d\big(g_1, p(g_2)\big)=\delta(g_1)$ (and making similar choices for $1_G$ and $g_2$), the geodesic triangle $p(g_1)\cup p(g_2) \cup p(g_1,g_2)$ is formed. 
Knowledge of the intervals ensures that minimal distances $\delta(1_G)$, $\delta(g_1)$ and $\delta(g_2)$ can be obtained. 

We now show that the interior must contain all solutions to the median problem for $1_G$, $g_1$ and $g_2$.

\begin{definition}[Steiner weight] 
The Steiner weight of a group element $h$, denoted $\sw(h)$, with respect to $1_G$, $g_1$ and $g_2$ is 
$$ \sw(h):=\displaystyle \sum_{\gamma \in \{1_{G},g_2,g_2\}} d(\gamma, h).  $$
\end{definition}

\begin{prop}\label{p:interior}
The interior of $1_G$, $g_1$ and $g_2$ contains all solutions to their median problem.
\end{prop}

\begin{proof}
Suppose that $y \in \Gamma(G,S) \setminus \ol{B}_{\delta(1_G)}(1_G)$  where $\ell(y) = \delta(1_G) + k$ for $k \in \N$ with $k > 0$. Without loss of generality, let $z$ be a vertex in $\Geo(g_1,g_2)$ such that $\ell(z) = \delta(1_G)$. Finally, let $x$ be a vertex on a geodesic path $p(y)$ where $\ell(x) = \delta(1_G)$ and $\ell(y) = \ell(x) + k$ so that $\ell(y) = \ell(z) + k$. 

By the triangle inequality, we have that $d(g_1,g_2) \leq d(g_1, y) + d(g_2,y)$. Therefore, 
$$
d(z,g_1) + d(z, g_2) \leq d(y, g_1) + d(y,g_2)
$$
since $d(g_1,g_2) = d(z,g_1) + d(z,g_2)$. Adding $\ell(z)$ to both sides and then $k$ to the right hand side only yields 
$$
d(z,g_1) + d(z, g_2) + \ell(z) < d(y, g_1) + d(y,g_2) + \ell(z) + k. 
$$
Since $\ell(y) = \ell(z) + k$, we have that $\sw(z) < \sw(y)$. The same argument holds when $y \in \Gamma(G,S) \setminus \ol{B}_{\delta(g_1)}(g_1)$ and $y \in \Gamma(G,S) \setminus \ol{B}_{\delta(g_2)}(g_2)$. Therefore, the interior must contain a solution to the median problem. Assuming there is a solution $a$ to the median problem outside of the interior yields a contradiction, since there must exist a point $b$ in the interior with $\sw(b) < \sw(a)$. Thus, the interior contains all solutions to the median problem.  
\end{proof}

The solution to the median problem may be not be unique. Here, we consider the Cayley graph of the symmetric group generated by the set of unsigned inversions $\{(1,2),\dots, (n,1)\}$. The presentation of the symmetric group under these relations is as follows \cite{egri2014group}.

\begin{equation*}
\begin{aligned}[c]
s_{i}^{2} = 1 & \, \, \, \, \text{for each $i = 1, \dots, n$;} \\
s_is_j=s_js_i & \, \, \, \, \text{if $(i-j)\mod n \neq \pm 1$;} \\
s_is_{i+1}s_i = s_{i+1}s_is_{i+1} &  \, \, \, \, \text{for each $i = 1, \dots, n-1$;} \\
s_n = s_{n-1}s_{n-2}\dots s_2s_1s_2\dots s_{n-2}s_{n-1}.
\end{aligned}
\end{equation*}
Note that the relations imply that the braid relation $s_ns_1s_n=s_1s_ns_1$ also applies to $s_n$ and $s_1$, and that all these relations preserve the \emph{parity} of the word-length.

In order to narrow the search for solutions to the median problem, we show that the distance between any two solutions in this Cayley graph must be even. 

\begin{prop}\label{p:even}
Let $1_G$, $g_1$ and $g_2$ be elements of the symmetric group with the above group presentation. If $m_1$ and $m_2$ are solutions to the median problem for $1_G$, $g_1$ and $g_2$, then $d(m_1,m_2)$ is even.
\end{prop}
\begin{proof}
Recall that the geodesic distance between $m_1$ and a group element $g$ is defined by $d(m_1,g)=|u|$,
where $u$ is a minimal length word in $S^*$ and $m_1 \overline{u} = g$. Let $u = s_1 \dots s_{2k+1}$ where $s_i \in S$ are inversions and $k \in \mathbb{N}\cup\{0\}$. We show that $m_1u \neq m_2$ for any such word $u$ and so $d(m_1,m_2)$ must be even. Considering the product $m_1s_1$, either $m_1s_1$ is a reduced word and $|m_1s_1| = |m_1|+1$ or $m_1s_1$ admits a reduction under the relations of the group's presentation. For any relation of the form $w_1 = w_2$ where $w_1$ and $w_2$ are words in $S^*$, the absolute value of $|w_1| - |w_2|$ is even since $w_1$ and $w_2$ are permutations of the same parity. Therefore, $\ell(m_1s_1)$ is either $\ell(m_1) + 1$ or $\ell(m_1) - \mathcal{O}_1$ where $\mathcal{O}_1$  is a positive odd number.  Similarly, $d(m_1s_1, g_1)$ is either $d(m_1, g_1) + 1$ or $d(m_1, g_1) - \mathcal{O}_2$ and $d(m_1s_1, g_2)$ is either $d(m_1, g_2) + 1$ or $d(m_1, g_2) - \mathcal{O}_3$. There are a number of values that $\sw(m_1s_1)$ can take given any combination of possible length changes. To summarise 
$$\sw(m_1s_1) =  \sw(m_1) + a - b$$
where $a \in \mathbb{N}\cup\{0\}$ and $a \leq 3$, and $b$ is a sum of odd numbers if $a < 3$. If $a = 3$, then $b=0$.  Extending this to an arbitrary odd number $2k+1$, we have that $\sw(m_1u) = \sw(m_1) +  a - b$ where $a-b$ is a sum of $3(2k+1)$ terms. If $a$ is odd, then $b$ is a summation consisting of an even number of positive odd integers. If $a$ is even, then $b$ is a summation consisting of an odd number of positive odd integers. In both instances, $a$ and $b$ have different parity. If $a > b$, then $\sw(m_1u) > \sw(m_1)$ and so $m_1u$ is not a median point. If $a < b$, then $\sw(m_1u) < \sw(m_1)$ contradicting the fact that $m_1$ is a median point. Therefore, $d(m_1, m_2)$ cannot be odd and so an even distance separates median points in this group presentation.  \end{proof}

\section{Algorithms and Computational issues}\label{s:algorithms}
Assuming that we can calculate the length efficiently, there is a straightforward algorithm for constructing the interval.
\begin{fact}$g_1\leq g_2 \Longleftrightarrow \ell(g_2)=\ell(g_1)+\ell(g_1^{-1}g_2).$
\end{fact}
For finding the geodesics, instead of a brute-force search in every direction, we can quickly discard those paths where the sum of the distance from the start and the remaining distance to the destination is more than the length of a shortest path.
This allows us to both study small (but non-trivial) cases of intervals in groups as well as to count/estimate the number of geodesics without constructing the interval explicitly.

\begin{algorithm}
\SetKwInOut{Input}{input}\SetKwInOut{Output}{output}
\Input{$g,h\in G$,\\$S$ generator set,\\$d$ distance function}
\Output{$[g,h]$ interval,\\$R_i$ rank-sets,\\$I_{g'}=\{(s,g's)\mid s\in S,\ g's\in [g,h]\}$}
\SetKwFunction{LC}{GradedInterval}
\SetKwInOut{Name}{\LC($g,h,S,d$)}
\SetKwData{Min}{min}
\SetKwFunction{Size}{Size}
\BlankLine
\Name{}
$n\leftarrow d(g,h)$\;
$R_0\leftarrow \{g\}$\;
\ForEach{$i\in \{1,\ldots,n\}$}{
  $R_i\leftarrow \varnothing$\;
  \ForEach{$g'\in R_{i-1}$}{
    $I_{g'}\leftarrow \varnothing$\;
    \ForEach{$s\in S$}{
      \If{$d(g's,h)=n-i$} {
        $R_i\leftarrow R_i\cup g's$\;
        $I_{g'}\leftarrow I_{g'}\cup (s,g's)$\;
      }
    }
  }
}
\caption{Constructing the graded interval \protect{$[g,h]$}.}
\label{alg:gradedinterval}
\end{algorithm}

In practice it could be unfeasible to fully calculate the graded interval.
Fortunately, Algorithm \ref{alg:gradedinterval} is iterative, so we can estimate the size of intervals by calculating only the first (and the last) $k$ grades. The reliability of these estimates depends on the particular generator set, whether the initial growth and the final shrinking rate of the intervals are proportional to their sizes.

For the median calculating algorithm we need to implement `geometrical' functions, like calculating spheres, balls, distances of points from intervals. Similar considerations apply here, one would like to calculate only the relevant `circular sector' guided by the distance from the other two corners of the triangle. 

The baseline versions of the above  algorithms are implemented in the \textsc{BioGAP} computer algebra package \cite{biogap} providing a solid base for experimenting with optimized algorithms.

\subsection*{Acknowledgements}{This research was supported by Australian Research Council grants DP130100248 and FT100100898.}

\bibliographystyle{plain}
\bibliography{caymedint}

\end{document}